\documentclass[reqno]{amsart}

\usepackage{amsmath,amsfonts,amssymb,amscd,verbatim,delarray,verbatim,epigraph}

\usepackage{amsthm}

\usepackage{url}

\newcommand{\F}{{\mathbb F}}

\newcommand{\Q}{{\mathbb Q}}

\begin{document}

\newtheorem{theorem}{Theorem}

\newtheorem{corollary}[theorem]{Corollary}
\newtheorem{corol}[theorem]{Corollary}
\newtheorem{conj}[theorem]{Conjecture}
\newtheorem{proposition}[theorem]{Proposition}

\theoremstyle{definition}
\newtheorem{defn}[theorem]{Definition}
\newtheorem{example}[theorem]{Example}

\newtheorem{remarks}[theorem]{Remarks}
\newtheorem{remark}[theorem]{Remark}

\newtheorem{question}[theorem]{Question}
\newtheorem{problem}[theorem]{Problem}

\newtheorem{quest}[theorem]{Question}
\newtheorem{questions}[theorem]{Questions}

\def\toeq{{\stackrel{\sim}{\longrightarrow}}}
\def\into{{\hookrightarrow}}


\def\alp{{\alpha}}  \def\bet{{\beta}} \def\gam{{\gamma}}
 \def\del{{\delta}}
\def\eps{{\varepsilon}}
\def\kap{{\kappa}}                   \def\Chi{\text{X}}
\def\lam{{\lambda}}
 \def\sig{{\sigma}}  \def\vphi{{\varphi}} \def\om{{\omega}}
\def\Gam{{\Gamma}}   \def\Del{{\Delta}}
\def\Sig{{\Sigma}}   \def\Om{{\Omega}}
\def\ups{{\upsilon}}


\def\F{{\mathbb{F}}}
\def\BF{{\mathbb{F}}}
\def\BN{{\mathbb{N}}}
\def\Q{{\mathbb{Q}}}
\def\Ql{{\overline{\Q }_{\ell }}}
\def\CC{{\mathbb{C}}}
\def\R{{\mathbb R}}
\def\V{{\mathbf V}}
\def\D{{\mathbf D}}
\def\BZ{{\mathbb Z}}
\def\K{{\mathbf K}}
\def\XX{\mathbf{X}^*}
\def\xx{\mathbf{X}_*}

\def\AA{\Bbb A}
\def\BA{\mathbb A}
\def\HH{\mathbb H}
\def\PP{\Bbb P}

\def\Gm{{{\mathbb G}_{\textrm{m}}}}
\def\Gmk{{{\mathbb G}_{\textrm m,k}}}
\def\GmL{{\mathbb G_{{\textrm m},L}}}
\def\Ga{{{\mathbb G}_a}}

\def\Fb{{\overline{\F }}}
\def\Kb{{\overline K}}
\def\Yb{{\overline Y}}
\def\Xb{{\overline X}}
\def\Tb{{\overline T}}
\def\Bb{{\overline B}}
\def\Gb{{\bar{G}}}
\def\Ub{{\overline U}}
\def\Vb{{\overline V}}
\def\Hb{{\bar{H}}}
\def\kb{{\bar{k}}}

\def\Th{{\hat T}}
\def\Bh{{\hat B}}
\def\Gh{{\hat G}}


\def\cC{{\mathcal C}}
\def\cU{{\mathcal U}}
\def\cP{{\mathcal P}}
\def\cV{{\mathcal V}}
\def\cS{{\mathcal S}}

\def\CG{\mathcal{G}}

\def\cF{{\mathcal {F}}}

\def\Xt{{\widetilde X}}
\def\Gt{{\widetilde G}}


\def\hh{{\mathfrak h}}
\def\lie{\mathfrak a}

\def\XX{\mathfrak X}
\def\RR{\mathfrak R}
\def\NN{\mathfrak N}

\def\minus{^{-1}}

\def\GL{\textrm{GL}}            \def\Stab{\textrm{Stab}}
\def\Gal{\textrm{Gal}}          \def\Aut{\textrm{Aut\,}}
\def\Lie{\textrm{Lie\,}}        \def\Ext{\textrm{Ext}}
\def\PSL{\textrm{PSL}}          \def\SL{\textrm{SL}} \def\SU{\textrm{SU}}
\def\loc{\textrm{loc}}
\def\coker{\textrm{coker\,}}    \def\Hom{\textrm{Hom}}
\def\im{\textrm{im\,}}           \def\int{\textrm{int}}
\def\inv{\textrm{inv}}           \def\can{\textrm{can}}
\def\id{\textrm{id}}              \def\Char{\textrm{char}}
\def\Cl{\textrm{Cl}}
\def\Sz{\textrm{Sz}}
\def\ad{\textrm{ad\,}}
\def\SU{\textrm{SU}}
\def\Sp{\textrm{Sp}}
\def\PSL{\textrm{PSL}}
\def\PSU{\textrm{PSU}}
\def\rk{\textrm{rk}}
\def\PGL{\textrm{PGL}}
\def\Ker{\textrm{Ker}}
\def\Ob{\textrm{Ob}}
\def\Var{\textrm{Var}}
\def\poSet{\textrm{poSet}}
\def\Al{\textrm{Al}}
\def\Int{\textrm{Int}}
\def\Smg{\textrm{Smg}}
\def\ISmg{\textrm{ISmg}}
\def\Ass{\textrm{Ass}}
\def\Grp{\textrm{Grp}}
\def\Com{\textrm{Com}}
\def\rank{\textrm{rank}}

\def\char{\textrm{char}}

\def\wid{\textrm{wd}}

\newcommand{\Or}{\operatorname{O}}

\def\tors{_\def{\textrm{tors}}}      \def\tor{^{\textrm{tor}}}
\def\red{^{\textrm{red}}}         \def\nt{^{\textrm{ssu}}}

\def\sss{^{\textrm{ss}}}          \def\uu{^{\textrm{u}}}
\def\mm{^{\textrm{m}}}
\def\tm{^\times}                  \def\mult{^{\textrm{mult}}}

\def\uss{^{\textrm{ssu}}}         \def\ssu{^{\textrm{ssu}}}
\def\comp{_{\textrm{c}}}
\def\ab{_{\textrm{ab}}}

\def\et{_{\textrm{\'et}}}
\def\nr{_{\textrm{nr}}}

\def\nil{_{\textrm{nil}}}
\def\sol{_{\textrm{sol}}}
\def\End{\textrm{End\,}}

\def\til{\;\widetilde{}\;}

\def\min{{}^{-1}}

\def\AGL{{\mathbb G\mathbb L}}
\def\ASL{{\mathbb S\mathbb L}}
\def\ASU{{\mathbb S\mathbb U}}
\def\AU{{\mathbb U}}


\title[Word maps  in Kac-Moody setting] {Word maps  in Kac-Moody setting}

\author[Klimenko, Kunyavski\u\i , Morita, Plotkin]{Elena Klimenko,
Boris Kunyavski\u\i , Jun Morita, Eugene Plotkin}

\address{Klimenko:  Mathematisches Institut,
 Heinrich-Heine-Universit\"at,
 Universit\"atsstr. 1,
 40225 D\"usseldorf,
 GERMANY}
\email{klimenko@hhu.de}

\address{Kunyavskii: Department of Mathematics,
Bar-Ilan University, 5290002 Ramat Gan, ISRAEL}
\email{kunyav@macs.biu.ac.il}

\address{Morita: Institute of Mathematics,
University of Tsukuba, 1-1-1 Tennodai,
Tsukuba, Ibaraki 305-8571, JAPAN}
\email{morita@math.tsukuba.ac.jp}

\address{Plotkin: Department of Mathematics,
Bar-Ilan University, 5290002 Ramat Gan, ISRAEL}
\email{plotkin@macs.biu.ac.il}

\begin{abstract}

The paper is a short survey of recent developments in the area of
word maps evaluated on groups and algebras. It is aimed to pose
questions relevant to Kac--Moody theory.
\end{abstract}

\maketitle





\epigraph{{\it Yukimi}\\
\smallskip
Iza yukan\\
Yukimi ni korobu\\
Tokoro made.\\
\bigskip

{\it Snow-viewing}\\
\smallskip
Well, let's go snow viewing,\\
Till we tumble over.

 {\it
Matsuo Bash\={o}, 1687}}


\medskip

\section{Word maps} \label{wordmaps}
These notes are devoted to problems arising from the general philosophy
of word maps. They consist of two sections. In the first one, we
describe some recent results and problems related to word maps on
simple algebraic groups and finite-dimensional Lie algebras. The
objective of the second section is to bring attention to Kac--Moody
groups and algebras, look at them from the viewpoint of word maps,
and formulate some problems.

The general setting is as follows. Let $\Theta$ be a variety of
algebras, $H$ be an algebra in $\Theta$, $W(X)$  be a free finitely
generated algebra in $\Theta$ with generators $x_1, \dots, x_n$. Fix
$w=w(x_1,x_2,\ldots,x_n)$ and consider the word map
\begin{equation} \label{wordmap}
w\colon H^n\to H.
\end{equation}
The map $w$ is the evaluation map: one substitutes $n$-tuples of
elements of the algebra $H$ instead of the variables and computes
the value by performing all algebra operations. Varying $\Theta$ we
arrive at problems on word maps specialized in a fixed $\Theta$.
Here we restrict ourselves to considering the varieties of groups
and Lie algebras.

The surjectivity of the word map for a given  word $w\in W(X)$ and an
algebra $H\in \Theta$ is traditionally the central question of the
theory. In other words,  the question is whether the equation
\begin{equation} \label{wordmapsur}
w(x_1,\ldots, x_n)=h
\end{equation}
has a solution for every $h\in H$. However, for many pairs of $w$
and $H$ the phenomenon of surjectivity is just a dream. So the
question about the width of $H$ with respect to a given $w$ (see the
definition below) often sounds more relevant.

Denote by $w(H)$ the value set of the map $w$ in the algebra $H$.
That is, an element $h\in H$ belongs to $w(H)$ if there exist
elements $h_1,h_2,\ldots, h_n$ of $H$ such that $w(h_1,h_2,\ldots,
h_n)=h$. We will be interested in estimating the size of  $w(H)$. A
reasonable but simpler problem is to compute the span
$\left<w(H)\right>$ of $w(H)$, where span means the verbal subgroup
generated by $w(H)$ in the group case and the linear span of $w(H)$
in the case of Lie algebras.
\medskip

Now assume that $\Theta$ is the variety of all groups. As mentioned above,
the surjectivity property of a word map is too strong for most pairs
$(w,H)$. Denote by $w(H)^k$ the set of elements $g\in H$ of the form
$g=h_1h_2\cdots h_k$ where $h_i\in w(H)$, $i=1,\ldots,k$.

The smallest $k$ such that $w(H)^k=H$ is called the $w$-width of the
group $H$. Denote it by $wd_w(H)$. If $H$ does not have finite
$w$-width, it is said to be of infinite $w$-width.

We want to estimate how much freedom we have in the general setting
of word maps for groups, and to choose a reasonable group $H$ where
word maps are evaluated. Let us look at two extreme classes of the
variety of all groups: free groups and simple groups. The essence of
word maps evaluated at these two poles is quite different.

\subsection{Free groups} Let $F_n(X)$, $X=\{x_1,\ldots,x_n\}$, be the
non-abelian free group with $n$ generators. If we take $w$ to be an
arbitrary word in $F_n(X)$, then common sense suggests that
$wd_w(F_n(X))$ should be infinite. An important theorem of
A.~Rhemtulla \cite{Rh} (see also \cite{Se}) confirms that this is
indeed the case for all non-universal words, see \cite{Se} for
details. Moreover, recently A.~Myasnikov and A.~Nikolaev proved the
following.

\begin{theorem} \cite{MN} \label{th:MN}
Let $H$ be a non-elementary hyperbolic group. Then the width
$wd_w(H)$ is infinite for each non-universal word $w$.
\end{theorem}
\noindent Since a ``random'' group is hyperbolic \cite{Ol}, for
``generic'' infinite groups the width of any non-universal word is
infinite. This result, though disappointing from the viewpoint of
word maps, is compensated by a tremendous theory of solutions of
equations in free and hyperbolic (non-elementary) groups. For
instance, A.~Mal'cev \cite{Mal} described the set of solutions of
the equation $[x,y]=[a,b]$, where $a$, $b$ are generators of the
free group and $[x,y]=xyx^{-1}y^{-1}$. L.~Comerford--C.~Edmunds
\cite{CE} and R.~Grigochuk--P.~Kurchanov \cite{GrK1}, \cite{GrK2}
described all solutions of quadratic equations in free groups.
Finally, G.~Makanin and A.~Razborov described solutions of an
arbitrary system of equations over a free group \cite{Mak},
\cite{Raz}. This theory leads to a developed geometry over free
groups and, in its turn, to solution of the famous Tarski problems
on elementary theory of free groups, see \cite{KhM1}--\cite{KhM6},
\cite{Se1}--\cite{Se9}, \cite{JS}, \cite{FGMRS}, etc.

The theory mentioned above goes far beyond the scope of these notes.
We will now turn to another pole of the variety of groups and
consider simple groups.

\subsection{Algebraic groups}
First of all, from the set of all words we choose the following representatives:

{\it Power map}:
\begin{equation} \label{1}
 w(x)=x^k.
\end{equation}

 {\it Commutator map}:
  \begin{equation} \label{2}
 w(x,y)=[x,y]=xyx^{-1}y^{-1}.
\end{equation}

 {\it Engel maps}: $w(x,y)=e_n(x,y).$
\begin{equation} \label{3}
  e_1(x,y)=[x,y],\dots, e_{n+1}(x,y)=[e_n,y]=[x,y,\ldots,y].
\end{equation}

{\it Quasi-Engel maps}: $w(x,y)=v_n(x,y).$
 \begin{equation} \label{4}
 v_1(x,y)=x^{-2}y^{-1}x,\dots,
v_{n+1}(x,y)=[xv_n(x,y)x^{-1},yv_n(x,y)y^{-1}].
\end{equation}

{\it Quasi-Engel maps}: $w(x,y)=s_n(x,y)$.
\begin{equation} \label{4'}
s_1(x,y)=x,\dots, s_{n+1}(x,y)=[y^{-1}s_n(x,y)y,s_n(x,y)^{-1}].
\end{equation}
Engel maps are related to nil elements and nilpotency, quasi-Engel
maps do the same with respect to solvability \cite{Z},
\cite{BGGKPP}, \cite{BWW}, \cite{GKP}.

\medskip

We start with a particular problem, which seems to be, at the moment,
the most challenging and tempting among the problems on word maps
for semisimple algebraic groups.

\begin{conj}\label{conj:1}
Let $G=PSL_2(\mathbb C)$, and let $w=w(x,y)$ be an arbitrary
non-identity word in $F(x,y)$. Then the word map $w\colon
PSL_2(\mathbb C)\times PSL_2(\mathbb C)\to PSL_2(\mathbb C)$ is
surjective. In other words, the equation
$$
w(x_1,x_2)=a
$$
has a solution for every $a\in PSL_2(\mathbb C)$.
\end{conj}

Although this conjecture is widely open, there are several partial
results. First of all, note that the power map $w=x^n$ is surjective
on $PSL_2(\mathbb C)$, since all roots are extractable in this group.
The next class of surjective words is given by

\begin{proposition}\label{prop:2}
Suppose that a word
$w(x,y)=x^{i_1}y^{j_1}x^{i_2}y^{j_2}\dots x^{i_k}y^{j_k}$
is not an identity of the infinite dihedral group $D_\infty$. Then
the word map $w(x,y)$ is surjective on $PSL_2(\mathbb C)$.
\end{proposition}

\begin{proof}
We can assume that $(i_1+\cdots +i_k)+(j_1+\cdots +j_k)=0$.
Otherwise, setting $x=y$, we arrive at the power word $x^m$. The
equation $x^m=c$ is solvable, since in $PSL_2(\mathbb C)$ one can
extract roots of an arbitrary degree. So, assuming the condition
above, we need to solve the equation $w(x,y)=c$ in  $PSL_2(\mathbb
C)$. Let us plug two involutions $a$, $b$ of $PSL_2(\mathbb C)$
into the word $w$. Since $a^{-1}=a$, $b^{-1}=b$, we have
$w(a,b)=(ab)^\ell$, $1\le\ell\le k$, or $w(a,b)=(ba)^\ell$, $1\le\ell\le k-1$. Note that
$ab\neq 1$ because $w$ is not an identity in $D_\infty$. Extracting
roots, we arrive at a system of equations:
\begin{equation} \label{ww}
a^2=1,\quad b^2=1,\quad ab=c',
\end{equation}
where  $c'$ is a prescribed element of $PSL_2(\mathbb C)$
(explicitly, $c'$ is either an $\ell$th root of $c$ or an $\ell$th
root of $bcb$). It remains to notice that in $PSL_2(\mathbb C)$
every element is a product of two involutions \cite{Ber}.
\end{proof}

\begin{corollary} All $n$-Engel maps are surjective on $PSL_2(\mathbb C)$.
\end{corollary}
\begin{proof}
Let $a,b\in PSL_2(\mathbb C)$ be involutions.
We have $e_1(a,b)=aba^{-1}b^{-1}=abab=(ab)^2$.
Show that $e_n(a,b)=(ab)^{2^n}$.  By induction,
$$
e_n(a,b)=[e_{n-1}(a,b),b]=[(ab)^{2^{n-1}},b]=(ab)^{2^{n-1}}b (ba)^{2^{n-1}}b=
$$
$$
=(ab)^{2^{n-1}}(ab)^{2^{n-1}}=(ab)^{2^n}.
$$
\end{proof}

Among other interesting word maps which fall under the conditions of
Proposition  \ref{prop:2} is the family $w_1(x,y)=[x,y]$,
$w_{n+1}(x,y)=[w_n(x,y),\delta_n(x,y)]$, where $\delta_n(x,y)=y$ if
$n=2k$ or $\delta_n(x,y)=x$ if $n=2k+1$. Moreover, T.~Bandman proved

\begin{theorem} \cite{Ba1}\label{ban:1}
The word map defined by  $w(x,y) \in F^{(1)} \setminus F^{(2)}$ is surjective on
$PSL_2(\mathbb C)$.
\end{theorem}
\noindent Here $F^{(1)}=[F,F]$ and $F^{(2)}=[F^{(1)}, F^{(1)}]$ are
the first and the second terms of the lower central series of the
two-generated free group $F=F(x,y)$. The proof  is reduced to
verification of the surjectivity of $w(x,y)$ on the unipotent
elements. Note that the surjectivity of $w(x,y)$ on the semisimple
elements of $PSL_2(\mathbb C)$ can be derived by computations of
trace polynomials (see, for example, \cite{Go}) and application of
methods of \cite{BGKJ}, \cite{BGaK}. Another proof of the
surjectivity of $w(x,y)$ on the semisimple elements can be found in
\cite{KBMR2}.

From Theorem \ref{ban:1} it follows that a key to solution of
Conjecture \ref{conj:1} lies in behaviour of words $w$ sitting in
$F^{(2)}$ or deeper in the lower central series. In this sense a
result of \cite{Ba1} stating  that the quasi-Engel maps
$w(x,y)=s_n(x,y)$ are surjective for any $n\geq 1$ is of promising
importance. Yet other computations from the same paper show that the
word
\begin{equation} \label{eq:f2}
w(x,y)= [[x, [x, y]], [y,[x, y]]]\in F^{(2)}
\end{equation}
is surjective on $PSL_2(\mathbb C)$ as well.

Conjecture \ref{conj:1} is a particular case of \cite[Question
2]{KBKP}. We repeat this question here:
\begin{problem}\label {pr:vestnik}
Let $\mathcal G$ be the class of simple groups $G$ of the form
$G=\mathbb G(k)$ where $k=\bar k$ is an algebraically closed field
and $\mathbb G$ is a semisimple adjoint linear algebraic group. Is
it true that word maps  are surjective for all non-power,
non-trivial words?
\end{problem}

It should be noted that the power map $w=x^n$ is surjective on an
arbitrary $G=\mathbb G(k)$ if $n$ is prime to 30 (see \cite{Stb},
\cite{Cha1}--\cite{Cha4}, \cite{Kun1}). Moreover, the only simple
algebraic adjoint group with extracting arbitrary roots is of type
A$_n$. Hence,  Problem \ref{pr:vestnik} for the groups of type A$_n$
can be reformulated as follows:

\begin{problem}\label{pr:slozhaja}
Is the word map $w\colon PSL_n(\mathbb C)\times PSL_n(\mathbb C)\to
PSL_n(\mathbb C)$ surjective for any non-trivial $w(x,y)\in
F_2(x,y)$?
\end{problem}
Despite many efforts applied to  Problem \ref{pr:vestnik} within
past years there are only few developments for specific word maps.
For instance, the commutator map $w(x,y)=[x,y]$ is surjective on
every semisimple adjoint linear algebraic group $G=\mathbb G(k)$
over an algebraically closed field $k$ \cite{Ree} (cf. the Ore
conjecture for finite simple groups: every element of a group is a
commutator). The next theorem concerning Engel words is due to
N.~Gordeev.

\begin{theorem} \cite{Gorde1} \label{th:q}
The image of the Engel map $e_n(x,y)$ on a semisimple adjoint linear
algebraic group $G=G(k)$, $k=\bar k$, contains all semisimple and
all unipotent elements.
\end{theorem}

So, for elements $g = tu$ that have the Jordan form with non-trivial
semisimple part $t$ and non-trivial unipotent part $u$, it is unclear
whether they are covered by $e_n(x,y)$. In the same paper it is proved
that the Engel map is surjective on $PSp_4(k)$ and $G_2(k)$. Summing up,
the surjectivity of Engel words is unknown for all groups except
$PSL_2(k)$, $PSp_4(k)$ and $G_2(k)$. For quasi-Engel maps the only
treated case is $PSL_2(k)$ \cite{Ba1}.

Now we are interested in estimating $wd_w(G)$, where $w(x_1,\ldots,
x_n)$ is any word in $F(X)$. The amazing theorem of A.~Borel
\cite{Bo} (see also \cite{La}, \cite{KBKP}) states that the image of
$w$ on every appropriate algebraic group is large in Zariski
topology.

\begin{theorem} \cite{Bo} \label{th:borel}
Let $w(x_1,\ldots, x_n)$ be a non-trivial word, let $G$ be a
connected semisimple linear algebraic $k$-group, and let $k$ be an
arbitrary field.  The image of the word map $w\colon G^n(k)\to G(k)$
is Zariski dense.
\end{theorem}

The next corollary confirms that for an algebraically closed field
the image is indeed big.

\begin{corollary}
Let $k$ be an algebraically closed field. Then $w(G(k))^2=G(k)$.
\end{corollary}
The latter fact means that the width of any word map evaluated on a
connected semisimple algebraic group over an algebraically closed
field is less than or equal to two.

\subsection{Algebraic groups. Forms} For an extended discussion of the
situation with word maps in this case see \cite{Kun1}. Here we just
mention a few principal results.

First of all, the power word map $x^n$ is surjective for the compact
real forms of all semisimple algebraic groups $G$ (cf. \cite{Cha4}).
The fact that the commutator map $w(x,y)=[x,y]$ is surjective for
the case of compact Lie groups $G=SU_n(\mathbb R)$ goes back to
H.~T\^oyama and M.~Got\^o. The same is true for $\mathbb
R$-anisotropic simple algebraic groups $G=G(\mathbb R)$, see
\cite{To}, \cite{Go}. In contrast with the split case, the behaviour
of Engel word maps on anisotropic groups is well-known due to
A.~Elkasapy--A.~Thom (cf. also \cite{Gorde1}):

\begin{theorem} \cite{ET} \label{th:thom}
Let $G$ be an $\mathbb R$-anisotropic simple algebraic group. Then
for every $n\ge 1$ the $n$-Engel word map $e_n(x,y)\colon G^2\to G$
is surjective.
\end{theorem}
\noindent In particular, $e_n$ is surjective on compact Lie groups
$SU_n(\mathbb R)$. Nothing is known about the behaviour of
quasi-Engel maps. The general result similar to that of Theorem
\ref{ban:1} was earlier established by A.~Elkasapy--A.~Thom
\cite{ET} for compact groups: if $w(x,y)$ does not belong to
$F^{(2)}$, then the corresponding word map is surjective on
$SU_n(\mathbb R)$.

\begin{remark}
The compact counterparts of Conjecture \ref{conj:1} and Problem
\ref{pr:vestnik} have negative answers \cite{Thom}. From the
construction of A.~Thom it also follows that the
width $wd_w(G)$ may be infinite in the compact case. The reason is that for some words
$w$ the image of the corresponding word map is a small neighbourhood
of the identity matrix, in particular, is not dense in Euclidean
topology. The authors are not aware of any example of a non-power
word $w$ (i.e., $w$ not representable as a proper power of another
word) such that the corresponding map on a simple real algebraic
group would be dominant in Euclidean topology but not surjective.
\end{remark}

\subsection{Finite simple groups} New achievements in the theory of
word maps evaluated on finite simple groups may be considered as a
sort of catalyst which gave rise to ongoing interest in this theory.

Among jewels of the last years was a positive solution of Ore's
problem: every element of a finite simple group is, indeed, a single
commutator (see \cite{Or}, \cite{EG} and the final solution in
\cite{LOST1}; a nice survey is given in \cite{Mall}).

A formidable progress in the description of images of word maps for
finite simple groups was obtained by M.~Larsen and A.~Shalev, who
stimulated a development of this area of research under the name of
``Waring type problems''. The final result of Larsen--Shalev--Tiep
is as follows:
\begin{theorem}[\cite{LST1}]\label{th:fg}
Let $w$ be an arbitrary non-trivial word of $F(x_1,\ldots, x_n)$.
There exists a constant $N = N(w)$ such that for all non-abelian
simple groups of order greater than $N$ one has
$$w(G)^2 = G.$$
\end{theorem}

So, for every $w$ the $w$-width of a finite simple group is less
than or equal to two. We believe that the coincidence of widths in
Borel's theorem for algebraic groups and the theorem of
Larsen--Shalev--Tiep quoted above might have a conceptual
explanation in the spirit of principles of model theory: loosely
speaking, a statement on algebraic varieties formulated over an
algebraically closed field should have a counterpart over every
sufficiently large finite field. Classical theorems of
Ax--Kochen--Grothendieck can serve as instances of such a principle,
see \cite{Ser}. This approach is awaiting an interested reader.
We also refer to the papers \cite{Sh1}--\cite{Sh3}, \cite{LST1},
\cite{LS1}, \cite{BGaK}, \cite{BGG}, \cite{KBKP}, \cite{NS}, \cite{MZ},
\cite{SW}, \cite{JLO}, \cite{KN}, \cite{Le}, etc., for details,
surveys and further explanations.

\medskip

Returning to specific word maps evaluated on finite simple groups,
even the case of Engel words $e_n(x,y)$ is still widely open. The
only known result is the surjectivity of $e_n$ on $PSL_2(\mathbb F_q)$
obtained in the paper of Bandman--Garion--Grunewald.
\begin{theorem} \cite{BGG} \label{th:bgg}
Let $G = PSL_2(\mathbb F_q)$. The $n$-th Engel word map $e_n(x,y)\colon
G^2\to G$ is surjective provided that $q > q_0(n)$. If $n\leq 4$,
then $e_n(x,y)$ is surjective for every~$q$.
\end{theorem}
\noindent The main ingredient of the proof is dynamics of trace
maps, see \cite{BGaK}.

\medskip

\subsection{Finite-dimensional simple Lie algebras} In this case
$W(X)$, $X=\{x_1,\ldots, x_n\}$, is the finite-dimensional free Lie
algebra, that is, the algebra of Lie polynomials with coefficients
from a field $k$. We will consider word maps $w\colon L^n\to  L$
where $w\in W(X)$, $L$ is a split semisimple Lie algebra.

Let us treat the brackets in maps (\ref{2}) as Lie operations. The
role of quasi-Engel maps (\ref{4}), (\ref{4'}) is played by the
sequence
\begin{equation}
v_1(x, y) = [x, y],\ldots , v_n(x, y) = [[v_{n-1}(x,y), x],
[v_{n-1}(x,y), y]],
\end{equation}
\noindent which is related to the solvability property of Lie
algebras \cite{BBGKP}, \cite{GKP}.

In this setting, many problems on word maps formulated above for
groups have solutions for Lie algebras.  For example, these are the
Ore problem \cite{Br} and the problem of the surjectivity of Engel
word maps \cite{BGKP}. The analogue of Conjecture~\ref{conj:1} has a
partial solution \cite{KBMR2}. Namely, the image of any non-trivial
homogeneous  $w\in W(X)$ evaluated on $\mathfrak sl_2(k)$, where $k$
is an algebraically closed field, is either~$0$, or the set of all
non-nilpotent traceless matrices, or $\mathfrak sl_2(k)$. Note that
the trivial image can appear, since for simple Lie algebras $w$ can
be an identity of $L$. Further, the Lie polynomial
\begin{equation}\label{pol}
v_2(x,y)=[[[x,y|, x], [[x,y], y]]
\end{equation}
covers only the non-nilpotent values on $\mathfrak sl_2(k)$
\cite{BGKP}, realizing the second possibility for images of an
arbitrary word $w$. In particular, this means that the quasi-Engel
word $w(x,y)$ is not surjective on $\mathfrak sl_2(k)$.  The proof
in the Lie case is a variation of a similar general result on
multilinear polynomials evaluated on simple associative algebras,
which is due to Kanel-Belov--Malev--Rowen:
\begin{theorem} \cite{KBMR1} \label{th:bmr}
If $w$ is a multilinear polynomial evaluated on the matrix ring
$M_2(k)$ (where $k$ is a quadratically closed field), then the image
of $w$ is either $\{0\}$, or $k$, or $\mathfrak sl_2(k)$, or the
full matrix algebra $M_2(k)$.
\end{theorem}
\noindent The question whether a multilinear Lie polynomial $w$ can
take only the values 0 or $\mathfrak sl_2(k)$ on  $\mathfrak
sl_2(k)$ is still open. We will finish this brief overview with
presenting a Lie-algebraic counterpart of Borel's theorem:

\begin{theorem} \cite{BGKP} \label{borel-lie}
Let $L$ be a split semisimple Lie algebra. Suppose that a Lie polynomial
$w(x_1,\ldots,x_n)$ is not an identity of the Lie algebra $\mathfrak
sl_2(k)$. Then the image of $w\colon L^n\to L$ is Zariski dense.
\end{theorem}

\begin{remark} \label{rem:liealg}
It is  unclear whether it is possible to improve the result by
suppressing the condition on $w(x_1,\ldots,x_n)$ in Theorem
\ref{borel-lie}. Also, the case of word maps evaluated on
non-classical Lie algebras is generally open (see \cite{KBKP} for
further discussions). If the base field $k$ is not algebraically
closed and $L$ is not split, the situation is not understood either.
Even for the simplest word $w=[x,y]$, even for $k=\mathbb R$, there
are only partial results: the corresponding map $L^2\to L$ is
surjective if $L$ is a compact simple algebra \cite{HM}, \cite{Ah}
(D.~Akhiezer \cite{Ah} also treats some non-compact algebras).
Surprisingly, for $w=[x,y]$ there is a very nice result of ``Waring
type'' due to G.~Bergman--N.~Nahlus \cite{BN} who use recent
two-generation theorems by J.-M.~Bois \cite{Boi}: if $L$ is a
finite-dimensional simple Lie algebra defined over any infinite
field of characteristic different from 2 and 3, then every $a\in L$
is a sum of two brackets: $a=[x,y]+[z,t]$. (Over $\mathbb R$, a
simple proof can be found in \cite{HM}; see \cite{Gorde} for the case
of arbitrary classical Lie algebras.) To the best of our
knowledge, there are no examples of simple Lie algebras of infinite
bracket width; moreover, we do not know any example where one
bracket is not enough.

In the case where the base field $k$ is replaced with some ring $R$,
the situation is almost totally unexplored, see \cite{Kun} for
discussion of some cases where $R$ is of arithmetic type.
\end{remark}


\section{Word maps for Kac--Moody groups and algebras. Problems}

Our next aim is to discuss how the results of the previous section
can be treated from the viewpoint of Kac--Moody groups and algebras.
Let $A = (a_{ij})$, $1\leq i,j\leq n$, be a generalized Cartan
matrix, that is, $a_{ii}=2$, $a_{ij}\leq 0$ if $i\neq j$ and
$a_{ij}=0$ if and only if $a_{ji}=0$. Every $A = (a_{ij})$ gives
rise to a complex Kac--Moody algebra $\mathfrak g_A$,  see
\cite{Mo}, \cite{Ka}, etc. According to the type of $A$
(positive-definite, positive-semidefinite, indefinite), we arrive at
finite-dimensional simple Lie algebras, affine Kac--Moody algebras,
and indefinite Kac--Moody algebras, respectively.

\subsection{Minimal (incomplete) Kac--Moody groups. Split case}
Given a generalized Cartan matrix $A$ and a field $k$ (or a ring
$R$), the value $G_A(k)$ of the Tits functor $G_A\colon \mathbb
Z$-$Alg\to Grp$ defines a minimal Kac--Moody group over $k$, see
\cite{Ti} (cf. \cite{Ca}, \cite{MR}). One can view this functor as a
generalization of the Chevalley--Demazure group scheme. We assume
that $A$ is indecomposable. As a rule, we assume that the functor
$G_A$ is simply connected. (However, speaking about the simplicity
of a Kac--Moody group $G_A(k)$ (resp. $k$-algebra $\mathfrak g_A$), we
will freely, often without special mentioning, use common language
abuse, assuming that we go over to its subquotient, taking the
derived subgroup (resp. subalgebra) and factoring out the centre, if
necessary.)

If $A$ is a definite matrix, the group $G_A(k)$ is a Chevalley group
$G_\Phi(k)$ where $\Phi$ is the root system corresponding to $A$.
Word maps arising in this case were considered in the previous
section.

If $A$ is of affine type, $G_A(k)$ is isomorphic to the Chevalley
group $G_\Phi(k[t,t^{-1}])$ where $k[t,t^{-1}]$ is the ring of
Laurent polynomials. Since $G_A(k)$ is a Chevalley group over a
ring, all Borel-type considerations are irrelevant. However, the
width of $G_A(k)$ with respect to natural words is unknown.
\begin{problem}
What is the width of the affine Kac--Moody group $G_A(k) \simeq
G_\Phi(k[t,t^{-1}])$ with respect to commutators? Is this width
finite?
\end{problem}

A comprehensive survey of the widths of the Chevalley groups over
rings can be found in \cite{HSVZ},  \cite{SV}, \cite{Ste}. We shall
quote some facts from these sources. First of all, since
$k[t,t^{-1}]$ is a Euclidean ring, the group $G_\Phi(k[t,t^{-1}])$
is perfect and, moreover, $G_\Phi(k[t,t^{-1}])=E_\Phi(k[t,t^{-1}])$
where $E_\Phi(k[t,t^{-1}])$ is the elementary subgroup. Further, it
is worth taking into account an unexpected  result from \cite{HSVZ}
which states that finite width in commutators is equivalent to
bounded generation in elementary unipotent elements (see \cite{Mr}
for a survey of the latter question). Furthermore, there is an
example by R.~K.~Dennis and L.~Vaserstein:
\begin{theorem}[\cite{DV}]
The group $SL(3,\mathbb {C}[t])$ does not have finite width with
respect to commutators.
\end{theorem}
\noindent This example shows that the commutator width for Chevalley
groups over polynomial rings with coefficients from $\mathbb {C}$ is
infinite. It also prevents from predictions in the case of Laurent
polynomials, without clarifying  the situation with usual polynomial
rings.

\begin{remark}
In general, the question whether a Chevalley group $G_\Phi(k[t])$
over a polynomial ring with coefficients in $k$ has a finite
commutator width should depend on the choice of the  ground field
$k$. Most likely, for fields of infinite transcendence degree over
the prime subfield the negative answer looks plausible. To the
contrary, for finite fields an affirmative answer sounds quite
reasonable. We quote \cite{Ste}: ``It is amazing that the answer is
unknown already for $SL_n(F[x])$, where $F$ is a finite field or a
number field'' and thank A.~Stepanov for the whole remark.
\end{remark}

Let now $G_A(k)$ be a Kac--Moody group of indefinite type. All these
groups are perfect for fields of size $>3$. 
Moreover, B.~R\'emy \cite{Remy1} and P.-E.~Caprace--B.~R\'emy \cite{CR}
showed that the minimal indefinite adjoint Kac--Moody groups
$G_A(F_q)$ are simple provided $q>n>2$ for every $A$ and for $n=2$ for
some $A$, J.~Morita and B.~R\'emy \cite{MoR} proved that in the case
where $k$ is the algebraic closure of $\mathbb F_q$ these groups are simple.
However, the simplicity problem for minimal groups over other fields
is open.

P.-E.~Caprace and K.~Fujiwara \cite{CF} showed that over finite
fields these (infinite) simple groups have infinite commutator
width. In general, especially over the $\mathbb C$ and $\mathbb R$,
problems on word maps are not yet settled.

\begin{problem}
Study the image of word maps for minimal non-affine Kac--Moody groups
over various fields. In particular, what can be said about the
commutator width and power width? Are there some density properties?
Are there analogues of results of Waring type for arbitrary words?
\end{problem}

The first step is to check what tools that proved to be useful
in finite-dimensional case can be extended to Kac--Moody setting.
This is not completely hopeless, as is shown in the papers
\cite{MP}, \cite{MP1}, where this was done for Gauss decomposition.

\medskip

\subsection{Maximal (complete) Kac--Moody groups. Split case}

Let $G_A(k)$ be an incomplete Kac--Moody group. There are several
ways to complete this group with respect to an appropriate topology.
Different methods of completions lead to very similar groups
(R\'emy--Ronan groups \cite{RR}, Mathieu groups \cite{Math},
Carbone--Garland groups \cite{CG}) whose distinctions are not very
important for our goals. It is known that the R\'emy--Ronan group is
smallest among them and can be obtained as a quotient of others. So,
once we are talking about a simple (indefinite) Kac--Moody complete
group we can assume that this is one of the incarnations of the
R\'emy--Ronan group. Denote it by $\overline G_A(k)$.
This 
group is simple (both as a topological group and an abstract group),
see \cite{Mo1}, \cite{KP}, \cite{Ga}, \cite{CER}, \cite{Mr1},
\cite{Rou1}, \cite{CR1}, \cite{Remy1}, etc.

Let first $\overline G_A(k)$ be a complete affine Kac--Moody group.
Then $\overline G_A(k)$ is isomorphic to a Chevalley group of the
form $G_\Phi(k((t)))$ where $k((t))$ is the field of formal Laurent
series over $k$. Thus, $\overline G_A(k)$ modulo centre is a simple
algebraic group. Here is an immediate corollary of Borel type:

\begin{corollary}\label{cor:aff}
Let $\overline G_A(k)$ be a complete adjoint affine Kac--Moody
group, and let $w=w(x_1,\dots ,x_n)$ be an arbitrary non-trivial
word. Then the image of $w\colon \overline G_A(k)^n\to \overline
G_A(k)$ is Zariski dense.
\end{corollary}

Note that the field  $K=k((t))$ is not algebraically closed.
Thus, it is not hard to prove that, say, if $w=x^n$ is a power word,
then the induced word map is not surjective, both for $G=SL_2$
defined over $k((t))$ and $k[t,t^{-1}]$: it is enough to look at the
equation $x^n=diag(t,t^{-1})$.

Further, Corollary \ref{cor:aff} does not provide estimates for
particular $wd_w(G_\Phi(K))$. However, recently Hui--Larsen--Shalev
\cite{HLS} proved that for an arbitrary $w$ we have
$wd_w(G_\Phi(K))\leq 4$ in every simple Chevalley group over any
infinite field $K$.  Thus, $wd_w(\overline G_A(k))\leq 4$. In a
similar spirit, from \cite{HLS} it follows that the width of
$\overline G_A(k)$ in squares equals 2. So, one can ask the
following question:

\begin{question}
What is the commutator width of $\overline G_A(k)$?
\end{question}

Let now $\overline G_A(k)$ be an arbitrary complete Kac--Moody
group. The problem whether a suitable variant of Borel's theorem is
valid for $\overline G_A(k)$ is among the most intriguing ones. For
example, let us view $\overline G_A(k)$ as a simple algebraic
ind-scheme \cite{Ku}, \cite{Math}, \cite{Rou} (see also \cite{AT},
\cite{KP1}, etc.). Then $\overline G_A(k)$ is endowed with a
Zariski-type topology \cite{Sha1}, \cite{Sha2}, \cite{Kam}. One of
the possible approaches is to test the image of a word map in this
topology.
\begin{problem}
Let $\overline G_A(k)$ be a complete non-affine Kac--Moody group,
and let $w=w(x_1,\dots ,x_n)$ be an arbitrary non-trivial word. Is
it true that $w((\overline G_A(k)^n)$ is dense in a Zariski
topology?
\end{problem}

Here are some related questions.

\begin{questions}
Is there a bound for $wd_w(\overline G_A(k))$? What is the
commutator width of $\overline G_A(k)$? How big is the image
$w((\overline G_A(k)^n)$ in the natural topology of $\overline
G_A(k)$?
\end{questions}

In a more general context, one can try to investigate an eventual
gap between the density and surjectivity properties:

\begin{question}
Do there exist a locally compact topological group $G$, simple at
least as a topological group, and a word $w=w(x_1,\dots ,x_n)$
non-representable as a proper power of another word such that the
corresponding word map $w\colon G^n\to G$ is not surjective but the
image of $w$ is dense?
\end{question}

Kac--Moody groups seem to be a sufficiently rich testing ground for
such sort of questions.

\medskip

\subsection{Forms of Kac--Moody groups} The theory of forms for
Kac--Moody groups is mainly developed in  \cite{Remy2},
\cite{Remy3}, \cite{Rou}, \cite{Rou1}, \cite{Rou2}, \cite{Hee},
\cite{Ra1}, \cite{Ra2}, etc.  All questions stated before in the
split case also make sense for twisted forms.  In particular, for
affine real forms classified in \cite{BMR}, this is related to
commutator widths in twisted Chevalley groups over the ring of Laurent
polynomials $R[t,t^{-1}]$. Another question to study in the compact case is the
phenomenon of Thom's sequence (see Section 1): can it occur for word
maps on any compact Kac--Moody group?

The case of integer forms of Kac--Moody groups, which were
extensively studied both from arithmetic viewpoint (Eisenstein
series, Langlands program) and with an eye towards applications in
mathematical physics (superstrings, supergravity), see, e.g.,
\cite{Ga1}, \cite{CG}, \cite{BC} and the references therein, is even
less explored. Any question on the word width of such groups seems
challenging. In a more group-theoretic spirit, one can mention
the papers of D.~Allcock and L.~Carbone \cite{Al}, \cite{AC}, who
established some finite presentability results in the affine and
hyperbolic cases.

\medskip

\subsection{Kac--Moody algebras}
As in Section 1, one can ask questions on the image of the map induced by a
Lie polynomial, as well as questions of width type, for Kac--Moody
algebras, {\it mutatis mutandis}. For example, it seems natural to
proceed in the spirit of Remark \ref{rem:liealg} and look at the
bracket width of simple (subquotients of) algebras $\mathfrak g_A$. Can it happen
that such an algebra is of infinite bracket width? Of bracket width greater than one?

In the case that for some simple algebra the bracket map turns out to be
surjective, one can study $n$-Engel maps ($n\ge 2$) and maps induced
by general multilinear Lie polynomials.


Another problem is to find out whether some counterpart of
Borel-type Theorem \ref{borel-lie} exists in Kac--Moody setting.

\bigskip

\noindent {\it Acknowledgements}. This research was supported by the
Israel Science Foundation, grant 1207/12. A substantial part of this
work was done during Klimenko's visits to the MPIM (Bonn) and
Bar-Ilan University. Kunyavski\u\i \ and Plotkin were supported in
part by the Minerva Foundation through the Emmy Noether Research
Institute of Mathematics. Morita was partially supported by the
Grants-in-Aid for Scientific Research (Monkasho Kakenhi) in Japan.
Support of these institutions is thankfully acknowledged.

Some questions raised in this paper were formulated after
discussions with N.~L.~Gordeev, L.~H.~Rowen, A.~Shalev,
A.~V.~Stepanov, N.~A.~Vavilov and E.~I.~Zelmanov to whom we are
heartily grateful.

\bigskip

{\it Three co-authors greet Jun Morita on his 60th birthday and wish him good health and activity.}



\begin{thebibliography}{999}

\bibitem{AT}
E.~Abe, M.~Takeuchi, {\it Groups associated with some types of
infinite-dimensional Lie algebras}, J. Algebra {\bf 146} (1992),
385--404.

\bibitem{Ah}
D.~Akhiezer, {\it On the commutator map for real semisimple Lie
algebras}, preprint, \url{arXiv:1501.02934v1} [math.RA], 2015, 6~pp.

\bibitem{Al}
D. Allcock, {\it Presentation of affine Kac--Moody groups over rings},
preprint, \url{arXiv:1409.0176v2} [math.GR], 2014, 22~pp.

\bibitem{AC}
D. Allcock, L. Carbone, {\it Presentation of  hyperbolic Kac--Moody
groups over rings}, preprint, \url{arXiv:1409.5918v1} [math.GR], 2014, 12~pp.

\bibitem{Ba1}
T.~Bandman, {\it Surjectivity of certain word maps on
$\PSL(2,\mathbb C)$ and $\SL(2,\mathbb C)$}, preprint,
\url{arXiv:1407.3447v2} [math.GR], 2014, 17 pp.

\bibitem{BBGKP}
T.~Bandman, M.~Borovoi, F.~Grunewald, B.~Kunyavski\u\i, E.~Plotkin,
{\it Engel-like characterization of radicals in finite dimensional
Lie algebras and finite groups}, Manuscripta Math. {\bf 119} (2006),
465--481.

\bibitem{BGG}
T.~Bandman, S.~Garion, F.~Grunewald, {\it On the surjectivity of
Engel words on $\PSL(2,q)$}, Groups Geom. Dyn. {\bf 6} (2012),
409--439.

\bibitem{BGaK}
T.~Bandman, S.~Garion, B.~Kunyavski\u\i, {\it Equations in simple
matrix groups: algebra, geometry, arithmetic, dynamics}, Cent. Eur.
J. Math. {\bf 12} (2014), 175--211.

\bibitem{BGKP}
T.~Bandman, N.~Gordeev, B.~Kunyavski\u\i, E.~Plotkin, {\it Equations
in simple Lie algebras}, J. Algebra {\bf 355} (2012), 67--79.

\bibitem{BGGKPP}
T.~Bandman, G.-M.~Greuel, F.~Grunewald, B.~Kunyavski\u\i,
G.~Pfister, E.~Plotkin, {\it Identities for finite solvable groups
and equations in finite simple groups}, Compos. Math. {\bf 142}
(2006), 734--764.

\bibitem{BGKJ}
T.~Bandman, F.~Grunewald, B.~Kunyavski\u\i \ (with an appendix by
N.~Jones), {\it Geometry and arithmetic of verbal dynamical systems
on simple groups}, Groups Geom. Dyn. {\bf 4} (2010), 607--655.

\bibitem{BC}
L. Bao, L. Carbone, {\it Integral forms of Kac--Moody groups and
Eisenstein series in low dimensional supergravity theories},
preprint, \url{arXiv:1308.6194v1} [hep-th], 2013, 46~pp.

\bibitem{Ber}
A.~F.~Beardon, {\it The geometry of discrete groups}, Graduate Texts
in Mathematics, vol.~91, Springer-Verlag, New York, 1983, xii+337
pp.

\bibitem{BMR}
H.~Ben Messaoud, G.~Rousseau, {\it Classification des formes
r\'eelles presque compactes des alg\`ebres de Kac--Moody affines},
J. Algebra {\bf 267} (2003), 443--513.

\bibitem{BN}
G. M. Bergman, N. Nahlus, {\it Homomorphisms on infinite direct product
algebras, especially Lie algebras}, J. Algebra {\bf 333} (2011), 67--104.

\bibitem{Boi}
J.-M. Bois, {\it Generators of simple Lie algebras in arbitrary
characteristics}, Math. Z. {\bf 262} (2009), 715--741.

\bibitem{Bo}
A.~Borel, {\it On free subgroups of semisimple groups}, Enseign.
Math. {\bf 29} (1983), 151--164; reproduced in {\OE}uvres: Collected
Papers, vol.~IV, Springer-Verlag, Berlin--Heidelberg, 2001, 41--54.

\bibitem{BWW}
J.~N.~Bray, J.~S.~Wilson, R.~A.~Wilson, {\it A characterization of
finite soluble groups by laws in two variables}, Bull. London Math.
Soc. {\bf 37} (2005), 179--186.

\bibitem{Br}
G.~Brown, {\it On commutators in a simple Lie algebra}, Proc. Amer.
Math. Soc. {\bf 14} (1963), 763--767.

\bibitem{CF}
P.-E.~Caprace, K.~Fujiwara, {\it Rank-one isometries of buildings
and quasi-morphisms of Kac--Moody groups}, Geom. Funct. Anal. {\bf
19} (2010), 1296--1319.

\bibitem{CR1}
P.-E.~Caprace, B.~R\'emy, {\it Simplicity and superrigidity of twin
building lattices}, Invent. Math. {\bf 176} (2009), 169--221.

\bibitem{CR}
P.-E.~Caprace, B.~R\'emy, {\it  Simplicity of twin tree lattices
with non-trivial commutation relations}, preprint,
\url{arXiv:1209.5372} [math.GR], 2012, 7 pp., to appear in Proc.
Ohio State U. Special Year on Geometric Group Theory, Springer LNM.

\bibitem{CER}
L.~Carbone, M.~Ershov, G.~Ritter, {\it Abstract simplicity of
complete Kac--Moody groups over finite fields}, J. Pure Appl.
Algebra {\bf 212} (2008), 2147--2162.

\bibitem{CG}
L.~Carbone, H.~Garland, {\it Existence of lattices in Kac--Moody
groups over finite fields}, Commun. Contemp. Math. {\bf 5} (2003),
813--867.

\bibitem{CG1}
L. Carbone, H. Garland, {\it Infinite dimensional Chevalley groups and
Kac--Moody groups over $\mathbb Z$}, preprint, 2013, 23 pp., available at
\url{http://www.math.rutgers.edu/~carbonel}

\bibitem{Ca}
R.~Carter, {\it Kac--Moody groups and their automorphisms}, Groups,
Combinatorics and Geometry (Durham, 1990) (M.~Liebeck, J.~Saxl,
eds.), London Math. Soc. Lecture Note Ser., vol.~165, Cambridge
Univ. Press, Cambridge, 1992, 218--228.

\bibitem{Cha1}
P.~Chatterjee, {\it On the surjectivity of the power maps of
algebraic groups in characteristic zero}, Math. Res. Lett. {\bf 9}
(2002), 741--756.

\bibitem{Cha2}
P.~Chatterjee, {\it On the surjectivity of the power maps of
semisimple algebraic groups}, Math. Res. Lett. {\bf 10} (2003),
625--633.

\bibitem{Cha3}
P.~Chatterjee, {\it On the power maps, orders and exponentiality of
$p$-adic algebraic groups}, J. reine angew. Math. {\bf 629} (2009),
201--220.

\bibitem{Cha4}
P.~Chatterjee, {\it Surjectivity of power maps of real algebraic
groups}, Adv. Math. {\bf 226} (2011), 4639--4666.

\bibitem{CE}
L.~P.~Comerford, C.~C.~Edmunds, {\it Solutions of equations in free
groups}, Group Theory (Singapore, 1987) (K.~N.~Cheng, Y.~K.~Leong,
eds.), de Gruyter, Berlin, 1989, 347--356.

\bibitem{DV}
R.~K.~Dennis, L.~N.~Vaserstein, {\it Commutators in linear groups},
$K$-Theory {\bf 2} (1989), 761--767.


\bibitem{ET}
A.~Elkasapy, A.~Thom, {\it About Got\^o's method showing
surjectivity of word maps}, Indiana Univ. Math. J. {\bf 63} (2014),
1553--1565.

\bibitem{EG}
E.~W.~Ellers, N.~Gordeev, {\it On the conjectures of J.~Thompson and
O.~Ore}, Trans. Amer. Math. Soc. {\bf 350} (1998), 3657--3671.

\bibitem{FGMRS}
B.~Fine, A.~Gaglione, A.~Myasnikov, G.~Rosenberger, D.~Spellman,
{\it The Elementary Theory of Groups. A Guide through the Proofs of
the Tarski Conjectures}, de Gruyter Expositions in Mathematics,
vol.~60, de Gruyter, Berlin, 2014, xiv+307 pp.

\bibitem{Ga1}
H. Garland, {\it The arithmetic theory of loop algebras}, J. Algebra
{\bf 53} (1978), 480--551; erratum {\it ibid.} {\bf 63} (1980), 285.

\bibitem{Ga}
H.~Garland, {\it The arithmetic theory of loop groups}. Inst. Hautes
\'Etudes Sci. Publ. Math. {\bf 52} (1980), 5--136.


\bibitem{Gorde}
N.~Gordeev, {\it Sums of orbits of algebraic groups I}, J. Algebra
{\bf 295} (2006), 62--80.

\bibitem{Gorde1}
N.~Gordeev, {\it On Engel words on simple algebraic groups}, J.
Algebra {\bf 425}  (2015), 215--244.

\bibitem{Go}
M.~Got\^o, {\it A theorem on compact semi-simple groups},
J. Math. Soc. Japan {\bf 1} (1949), 270--272.

\bibitem{GrK1}
R.~I.~Grigorchuk, P.~F.~Kurchanov, {\it On the width of elements in
free groups} (Russian), Ukrain. Mat. Zh. {\bf 43} (1991), no.~7--8,
911--918; English transl. in Ukrainian Math. J. {\bf 43} (1991), no.
7--8, 850--856.

\bibitem{GrK2}
R.~I.~Grigorchuk, P.~F.~Kurchanov, {\it Some questions of group
theory related to geometry}, Encyclopaedia Math. Sci., vol.~58,
Algebra, VII, Springer, Berlin, 1993, 167--232, 233--240.

\bibitem{GKP}
F. Grunewald, B. Kunyavski\u\i, E. Plotkin, {\it Characterization of
solvable groups and solvable radical}, Internat. J. Algebra Comput.
{\bf 23} (2013), 1011--1062.

\bibitem{HSVZ}
R.~Hazrat, A.~Stepanov, N.~Vavilov, Z.~Zhang, {\it Commutator width
in Chevalley groups}, Note Mat. {\bf 33} (2013), 139--170.

\bibitem{Hee}
J.-Y.~H\'ee, {\it Sur la torsion de Steinberg--Ree des groupes de
Chevalley et de Kac--Moody}, Th\`ese d'Etat Univ. Paris 11, 1993.

\bibitem{HM}
K. H. Hofmann, S. A. Morris, {\it The Lie Theory of Connected Pro-Lie
Groups. A Structure Theory for Pro-Lie Algebras, Pro-Lie Groups, and
Connected Locally Compact Groups}, EMS Tracts Math., vol.~2, Eur.
Math. Soc., Z\"urich, 2007, xvi+678 pp.

\bibitem{HLS}
C.-Y.~Hui, M.~Larsen, A.~Shalev, {\it The Waring problem for Lie
groups and Chevalley groups}, preprint, \url{arXiv:1404.4786v1}
[math.GR], 2014, 17 pp., to appear in Israel J. Math.

\bibitem{JS}
E.~Jaligot, Z.~Sela, {\it Makanin--Razborov diagrams over free
products}, Illinois J. Math. {\bf 54} (2010), 19--68.

\bibitem{JLO}
S.~Jambor, M.~W.~Liebeck, E.~A.~O'Brien, {\it Some word maps that
are non-surjective on infinitely many finite simple groups},  Bull.
London Math. Soc. {\bf 45} (2013), 907--910.

\bibitem{Ka}
V.~G.~Kac, {\it Infinite-dimensional Lie Algebras}, 3rd ed.,
Cambridge Univ. Press, Cambridge, 1990, xxii+400 pp.

\bibitem{KP1}
V.~G.~Kac, D.~H.~Peterson, {\it Regular functions on certain
infinite-dimensional groups}, Arithmetic and Geometry, Vol. II
(M.~Artin, J.~Tate, eds.), Progr. Math., vol. 36, Birkh\"auser,
Boston, MA, 1983, 141--166.

\bibitem{KP}
V.~G.~Kac, D.~H.~Peterson,  {\it Defining relations of certain
infinite-dimensional groups}, The Mathematical Heritage of \'Elie
Cartan (Lyon, 1984), Ast\'erisque 1985, Num\'ero Hors S\'erie,
165--208.

\bibitem{Kam}
T.~Kambayashi, {\it Some basic results on pro-affine algebras and
ind-affine schemes}, Osaka J. Math. {\bf 40} (2003), 621--638.

\bibitem{KBKP}
A.~Kanel-Belov, B.~Kunyavski\u\i, E.~Plotkin, {\it Word equations in
simple groups and polynomial equations in simple algebras}, Vestnik
St. Petersburg Univ. Math. {\bf 46} (2013), no.~1, 3--13.

\bibitem{KBMR1}
A.~Kanel-Belov, S.~Malev, L.~Rowen, {\it The images of
non-commutative polynomials evaluated on $2\times 2$ matrices},
Proc. Amer. Math. Soc. {\bf 140} (2012), 465--478.

\bibitem{KBMR2}
A.~Kanel-Belov, S.~Malev, L.~Rowen, {\it The images of Lie
polynomials evaluated on $2\times 2$ matrices over an algebraically
closed field}, preprint, 2015.

\bibitem{KN}
M.~Kassabov, N.~Nikolov, {\it Words with few values in finite simple
groups},  Quart. J. Math. {\bf 64} (2013), 1161--1166.

\bibitem{KhM1}
O.~Kharlampovich, A.~Myasnikov, {\it Irreducible affine varieties
over a free group. I. Irreducibility of quadratic equations and
Nullstellensatz}, J. Algebra {\bf 200} (1998), 472--516.

\bibitem{KhM2}
O.~Kharlampovich, A.~Myasnikov, {\it Irreducible affine varieties
over a free group. II. Systems in triangular quasi-quadratic form
and description of residually free groups}, J. Algebra {\bf 200}
(1998), 517--570.

\bibitem{KhM3}
O.~Kharlampovich, A.~Myasnikov, {\it Algebraic geometry over free
groups: lifting solutions into generic points},  Groups, Languages,
Algorithms (A.~V.~Borovik, ed.), Contemp. Math., vol.~378, Amer.
Math. Soc., Providence, RI, 2005, 213--318.

\bibitem{KhM4}
O.~Kharlampovich, A.~Myasnikov, {\it Elementary theory of free
non-abelian groups}, J. Algebra {\bf 302} (2006), 451--552.

\bibitem{KhM5}
O.~Kharlampovich, A.~Myasnikov, {\it Equations and fully residually
free groups}, Combinatorial and Geometric Group Theory
(O.~Bogopolski {\it et al.}, eds.), Trends Math.,
Birkh\"auser/Springer Basel AG, Basel, 2010, 203--242.

\bibitem{KhM6}
O.~Kharlampovich, A.~Myasnikov, {\it Equations and Algorithmic
Problems in Groups}, Publica\c{c}\~{o}es Matem\'aticas do IMPA.
[IMPA Mathematical Publications] XX Escola de \'Algebra. [XX School
of Algebra] Instituto Nacional de Matem\'atica Pura e Aplicada
(IMPA), Rio de Janeiro, 2008, ii+37 pp.

\bibitem{Ku}
S.~Kumar, {\it Kac--Moody Groups, Their Flag Varieties and
Representation Theory}, Progr. Math., vol.~204, Birkh\"auser,
Boston, MA, 2002, xvi+606 pp.

\bibitem{Kun}
B. Kunyavski\u\i, {\it Equations in matrix groups and algebras over
number fields and rings: prolegomena to a lowbrow noncommutative
Diophantine geometry}, Arithmetic and Geometry (L.~Dieulefait
{\it et al.}, eds.), London Math. Soc. Lecture Note Ser., vol.~420,
Cambridge Univ. Press, Cambridge, 2015, 264--282.

\bibitem{Kun1}
B.~Kunyavski\u\i, {\it Complex and real geometry of word equations
in simple matrix groups and algebras}, preprint, 2015, 10 pp.

\bibitem{La}
M.~Larsen, {\it Word maps have large image}, Israel J. Math. {\bf
139} (2004), 149--156.

\bibitem{LS1}
M.~Larsen, A.~Shalev, {\it Word maps and Waring type problems}, J.
Amer. Math. Soc. {\bf 22} (2009), 437--466.

\bibitem{LST1}
M.~Larsen, A.~Shalev, P.~H.~Tiep, {\it The Waring problem for finite
simple groups}, Ann. Math. {\bf 174} (2011), 1885--1950.

\bibitem{Le}
M.~Levy, {\it Word maps with small image in simple groups},
preprint, \url{arXiv:1206.1206v1} [math.GR], 2012, 5 pp.

\bibitem{LOST1}
M.~W.~Liebeck, E.~A.~O'Brien, A.~Shalev, P.~H.~Tiep, {\it The Ore
conjecture}, J. Eur. Math. Soc. {\bf 12} (2010), 939--1008.

\bibitem{LOST2}
M.~W.~Liebeck, E.~A.~O'Brien, A.~Shalev, P.~H.~Tiep, {\it
Commutators in finite quasisimple groups}, Bull. Lond. Math. Soc.
{\bf 43} (2011), 1079--1092.

\bibitem{Mak}
G.~S.~Makanin, {\it Equations in a free group} (Russian), Izv. Akad.
Nauk SSSR Ser. Mat. {\bf 46} (1982), 1199--1273; English transl. in:
Math. USSR-Izv. {\bf 21} (1983), 483--546.

\bibitem{Mal}
A.~I.~Mal'cev, {\it On the equation
$zxyx^{-1}y^{-1}z^{-1}=aba^{-1}b^{-1}$ in a free group}. (Russian)
Algebra i Logika Sem. {\bf 1} (1962), no. 5, 45--50.

\bibitem{Mall}
G. Malle, {\it The proof of Ore's conjecture [after Ellers--Gordeev and
Liebeck--O'Brien--Shalev--Tiep]},  S\'em. Bourbaki, Ast\'erisque {\bf 361}
(2014), Exp.~1069, 325--348.

\bibitem{Mr1}
T.~Marquis, {\it  Abstract simplicity of locally compact Kac--Moody
groups}. Compos. Math. {\bf 150} (2014), 713--728.

\bibitem{MZ}
C.~Martinez, E.~Zelmanov, {\it Products of powers in finite simple
groups}, Israel J. Math. {\bf 96} (1996), 469--479.

\bibitem{Math}
O.~Mathieu, {\it Construction d'un groupe de Kac--Moody et
applications}, Compos. Math. {\bf 69} (1989), 37--60.

\bibitem{Mo}
R.~V.~Moody, {\it A new class of Lie algebras}, J. Algebra {\bf 10}
(1968), 211--230.

\bibitem{Mo1}
R.~V.~Moody, {\it A simplicity theorem for Chevalley groups defined
by generalized Cartan matrices}, preprint, 1982.

\bibitem{Mor}
J.~Morita, {\it On adjoint Chevalley groups associated with
completed Euclidean Lie algebras}, Comm. Algebra {\bf 12} (1984),
673--690.

\bibitem{MP}
J.~Morita, E.~Plotkin, {\it Gauss decompositions of Kac--Moody
groups}, Comm. Algebra {\bf 27} (1999), 465--475.

\bibitem{MP1}
J.~Morita, E.~Plotkin, {\it Prescribed Gauss decompositions for
Kac--Moody groups over fields}, Rend. Sem. Mat. Univ. Padova
{\bf 106} (2001), 153--163.

\bibitem{MR}
J.~Morita, U.~Rehmann, {\it A Matsumoto-type theorem for Kac--Moody
groups}, Tohoku Math. J. {\bf 42} (1990), 537--560.

\bibitem{MoR}
J.~Morita, B.~R\'emy, {\it Simplicity of some twin tree automorphism
groups with trivial commutation relations}, Canad. Math. Bull. {\bf
57} (2014), 390--400.

\bibitem{Mr}
D.~Morris, {\it Bounded generation of $\SL(n,A)$ (after D.~Carter,
G.~Keller, and E.~Paige)}, New York J. Math. {\bf 13} (2007),
383--421.

\bibitem{MN}
A.~Myasnikov, A.~Nikolaev, {\it Verbal subgroups of hyperbolic
groups have infinite width}, J. Lond. Math. Soc. {\bf 90} (2014),
573--591.

\bibitem{NS}
N.~Nikolov, D.~Segal, {\it Powers in finite groups}, Groups Geom.
Dyn. {\bf 5} (2011), 501--507.

\bibitem{Ol}
A. Yu. Ol'shanski\u\i , {\it Almost every group is hyperbolic}, Internat.
J. Algebra Comput. {\bf 2} (1992), 1--17.

\bibitem{Or}
O.~Ore, {\it Some remarks on commutators}, Proc. Amer. Math. Soc.
{\bf 2} (1951), 307--314.

\bibitem{Ra1}
J.~Ramagge, {\it On certain fixed point subgroups of affine
Kac--Moody groups}, J. Algebra {\bf 171} (1995), 473--514.

\bibitem{Ra2}
J.~Ramagge, {\it A realization of certain affine Kac--Moody groups
of types II and III}, J. Algebra {\bf 171} (1995), 713--806.

\bibitem{Raz}
A.~A.~Razborov, {\it Systems of equations in a free group} (Russian),
Izv. Akad. Nauk SSSR Ser. Mat. {\bf 48} (1984), 779--832; English
transl. in: Math. USSR-Izv. {\bf 25} (1985), 115--162.


\bibitem{Ree}
R.~Ree, {\it Commutators in semi-simple algebraic groups}, Proc.
Amer. Math. Soc. {\bf 15} (1964), 457--460.

\bibitem{Remy2}
B.~R\'emy, {\it Groupes de Kac--Moody d\'eploy\'es et presque
d\'eploy\'es}, Ast\'erisque No. {\bf 277} (2002), viii+348 pp.

\bibitem{Remy3}
B.~R\'emy, {\it Kac--Moody groups: split and relative theories.
Lattices}, Groups: Topological, Combinatorial and Arithmetic Aspects
(T.~W.~M\"uller, ed.), London Math. Soc. Lecture Note Ser.,
vol.~311, Cambridge Univ. Press, Cambridge, 2004, 487--541.

\bibitem{Remy1}
B.~R\'emy (with an appendix by P.~Bonvin), {\it Topological
simplicity, commensurator super-rigidity and non-linearities
of Kac--Moody groups}, Geom. Funct. Anal. {\bf 14} (2004),
810--852.

\bibitem{RR}
B.~R\'emy, M.~Ronan, {\it Topological groups of Kac--Moody type,
right-angled twinnings and their lattices}, Comment. Math. Helv.
{\bf 81} (2006), 191--219.

\bibitem{Rh}
A.~H.~Rhemtulla, {\it A problem of bounded expressibility in free
products}, Proc. Cambridge Philos. Soc. {\bf 64} (1968), 573--584.

\bibitem{Rou2}
G.~Rousseau, {\it Almost split $K$-forms of Kac--Moody algebras},
Infinite-dimensional Lie Algebras and Groups (Luminy-Marseille,
1988) (V.~G.~Kac, ed.), Adv. Ser. Math. Phys., vol.~7, World Sci.
Publ., Teaneck, NJ, 1989, 70--85.

\bibitem{Rou}
G.~Rousseau, {\it On forms of Kac--Moody algebras}, Algebraic Groups
and Their Generalizations: Quantum and Infinite-dimensional Methods
(University Park, PA, 1991) (W.~J.~Haboush, B.~J. Parshall, eds.),
Proc. Sympos. Pure Math., vol.~56, Part 2, Amer. Math. Soc., Providence,
RI, 1994, 393--399.

\bibitem{Rou1}
G.~Rousseau, {\it Groupes de Kac--Moody d\'eploy\'es sur un corps
local, II. Masures ordonn\'ees,} preprint, \url{arXiv:1009.0138v2}
[math.GR], 2012, 61 pp.

\bibitem{SW}
J.~Saxl, J.~S.~Wilson, {\it A note on powers in simple groups},
Math. Proc. Cambridge Philos. Soc. {\bf 122} (1997), 91--94.

\bibitem{Se}
D.~Segal, {\it Words: Notes on Verbal Width in Groups}, London Math.
Soc. Lecture Note Ser., vol.~361, Cambridge Univ. Press, Cambridge,
2009, xii+121 pp.

\bibitem{Se1}
Z.~Sela, {\it Diophantine geometry over groups I: Makanin-Razborov
diagrams}, Inst. Hautes \'Etudes Sci. Publ. Math. {\bf 93} (2001),
31--105.

\bibitem{Se2}
Z.~Sela, {\it Diophantine geometry over groups II: Completions,
closures and formal solutions}, Israel J. Math. {\bf 134} (2003),
173--254.

\bibitem{Se3}
Z.~Sela, {\it Diophantine geometry over groups III: Rigid and solid
solutions}, Israel J. Math. {\bf 147} (2005), 1--73.

\bibitem{Se4}
Z.~Sela, {\it Diophantine geometry over groups IV: An iterative
procedure for validation of a sentence}, Israel J. Math. {\bf 143}
(2004), 1--130.

\bibitem{Se5}
Z.~Sela, {\it Diophantine geometry over groups V: Quantifier
elimination I, II}, Israel J. Math. {\bf 150} (2005), 1--197; Geom.
Funct. Anal. {\bf 16} (2006), 537--706.

\bibitem{Se6}
Z.~Sela, {\it Diophantine geometry over groups VI: The elementary
theory of a free group}, Geom. Funct. Anal. {\bf 16} (2006),
707--730.

\bibitem{Se7}
Z.~Sela, {\it Diophantine geometry over groups VII: The elementary
theory of a hyperbolic group}, Proc. Lond. Math. Soc. {\bf 99}
(2009), 217--273.

\bibitem{Se8}
Z.~Sela, {\it Diophantine geometry over groups VIII: Stability},
Ann. Math. {\bf 177} (2013), 787--868.

\bibitem{Se9}
Z.~Sela, {\it Diophantine geometry over groups X: The elementary
theory of free products of groups}, preprint,
\url{arXiv:1012.0044v1} [math.GR], 2010, 143 pp.

\bibitem{Ser}
J-P. Serre, {\it How to use finite fields for problems concerning
infinite fields}, Arithmetic, Geometry, Cryptography and Coding
Theory (G.~Lachaud, C.~Ritzenthaler, M.~A.~Tsfasman, eds.),
Contemp. Math., vol.~487, Amer. Math. Soc., Providence, RI, 2009,
183--193.

\bibitem{Sha1}
I.~R.~Shafarevich, {\it On some infinite-dimensional groups}, Rend.
Mat. e Appl. {\bf 25} (1966), no.~1--2, 208--212.

\bibitem{Sha2}
I.~R.~Shafarevich, {\it On some infinite-dimensional groups. II}
(Russian), Izv. Akad. Nauk SSSR Ser. Mat. {\bf 45} (1981), 214--226,
240; English transl. in: Math. USSR-Izv. {\bf 18} (1982), 185--194.

\bibitem{Sh1}
A.~Shalev, {\it Commutators, words, conjugacy classes and character
methods}, Turkish J. Math. {\bf 31} (2007), suppl., 131--148.

\bibitem{Sh2}
A.~Shalev, {\it Word maps, conjugacy classes, and a non-commutative
Waring-type theorem}, Ann. Math.  {\bf 170} (2009), 1383--1416.

\bibitem{Sh3}
A.~Shalev, {\it  Some results and problems in the theory of word
maps}, Erd\"os Centennial (L.~Lov\'asz, I.~Z.~Ruzsa, V.~T.~S\'os,
eds.), Bolyai Soc. Math. Stud., vol.~25, J\'anos Bolyai Math. Soc.,
Budapest, 2013, 611--649.


\bibitem{Sta}
I.~Stampfli, {\it On the topologies on ind-varieties and related
irreducibility questions}, J. Algebra {\bf 372} (2012), 531--541.

\bibitem{Stb}
R.~Steinberg, {\it On power maps in algebraic groups}, Math. Res.
Lett. {\bf 10} (2003), 621--624.

\bibitem{Ste}
A.~Stepanov, {\it Structure of Chevalley groups over rings via
universal localization}, preprint, \url{arXiv:1303.6082v3}
[math.RA], 2013, 18 pp., to appear in J. $K$-Theory.

\bibitem{SV}
A.~Stepanov, N.~Vavilov, {\it On the length of commutators in
Chevalley groups}, Israel J. Math. {\bf 185} (2011), 253--276.

\bibitem{Thom}
A.~Thom, {\it Convergent sequences in discrete groups}, Canad. Math.
Bull. {\bf 56} (2013), 424--433.

\bibitem{Ti}
J.~Tits, {\it Uniqueness and presentation of Kac--Moody groups over
fields}, J. Algebra {\bf 105} (1987), 542--573.

\bibitem{To}
H.~T\^oyama, {\it On commutators of matrices}, K\=odai Math. Sem.
Rep. {\bf 1} (1949), no.~5--6, 1--2.

\bibitem{Z}
M.~Zorn, {\it Nilpotency of finite groups}, Bull. Amer. Math. Soc.
{\bf 42} (1936), 485--486.




\end{thebibliography}
\end{document}